\documentclass{amsart}
\usepackage{amssymb}
\usepackage{amsthm}
\usepackage{epsfig}
\usepackage{color}
\usepackage{url}

\usepackage[all]{xy}

\newcommand{\mc}{\mathcal}
\newcommand{\Vol}{\mathrm{Vol}}

\newcommand{\sign}{\mathrm{sign}}

\newcommand{\algvol}{\mathrm{algvol}}
\newcommand{\leftquotient}{\setminus}
\newcommand{\bR}{\mathbb{R}}
\newcommand{\bQ}{\mathbb{Q}}
\newcommand{\bZ}{\mathbb{Z}}
\newcommand{\bH}{\mathbb{H}}
\newcommand{\bN}{\mathbb{N}}
\newcommand{\Id}{{\mathbf{1}}}
\newcommand{\bv}{\mathbf{v}}

\newcommand\co{\colon\thinspace}

\newtheorem{theorem}{Theorem}[section]
\newtheorem{lemma}[theorem]{Lemma}
\newtheorem{corollary}[theorem]{Corollary}
\newtheorem{proposition}[theorem]{Proposition}
\newtheorem{question}[theorem]{Question}
\newtheorem{conjecture}[theorem]{Conjecture}

\newtheorem{claim}[theorem]{Claim}

\theoremstyle{definition}
\newtheorem{remark}[theorem]{Remark}
\newtheorem{definition}[theorem]{Definition}

\begin{document}

\author[K. Fujiwara]{Koji Fujiwara}
\address{Tohoku University}
\email{fujiwara@math.is.tohoku.ac.jp}

\author[J. F. Manning]{Jason Fox Manning}
\address{University at Buffalo, SUNY}
\email{j399m@buffalo.edu}

\title{Simplicial volume and fillings of hyperbolic manifolds}

\thanks{Koji Fujiwara is supported in part by a Grant-in-Aid for Scientific
Research (No. 19340013).  
Jason Manning is supported in part by NSF grant
DMS-0804369.}

\begin{abstract}
Let $M$ be a hyperbolic $n$--manifold whose cusps have torus
cross-sections.  
In \cite{FM1}, the authors constructed a variety of
nonpositively and negatively curved spaces as ``$2\pi$--fillings'' of
$M$ by replacing the cusps of $M$ with compact  ``partial cones'' of
their boundaries.
These $2\pi$--fillings are closed pseudomanifolds, and so 
have a fundamental class.  We show that the simplicial volume of 
any such $2\pi$--filling is positive, and bounded above by
$\frac{\mathrm{Vol}(M)}{v_n}$, where $v_n$ is the volume of a regular
ideal hyperbolic $n$--simplex.  This result generalizes the fact that
hyperbolic Dehn filling of a $3$--manifold does not increase
hyperbolic volume. 

In particular, we obtain information about the simplicial volumes of
some $4$--dimensional homology spheres described by Ratcliffe and
Tschantz, answering a question of Belegradek and establishing 
the existence of $4$--dimensional homology spheres with positive
simplicial volume.  
\end{abstract} 

\maketitle


\section{Introduction}
Simplicial volume was defined for manifolds by Gromov in
\cite{gromov:vbc}.
For an $n$-dimensional closed manifold $M$ of 
constant sectional curvature $K=-1$, 
the simplicial volume $\|M\|$ is proportional to the 
(Riemannian) volume $\Vol(M)$,
and the ratio is the volume of a regular ideal hyperbolic
$n$--simplex; in particular, it
depends only on the dimension
\cite[0.4]{gromov:vbc}.
More generally, \cite[0.3,0.5]{gromov:vbc}, for constants $a\le b< 0$, there
is a constant $C=C(a,b,n)>1$ so that if the sectional curvatures $K$ satisfy
$a \le K \le b$, then
\begin{equation}\label{propclosed}
 \Vol(M)/C \le \|M\| \le C\Vol(M).
\end{equation}

This note is motivated in part by the (apparently open) question:
\begin{question}\label{q:manif}
  Fix $v>0$, and let $\mathcal{M}_{n,v}$ be the set
  of (homotopy classes of) closed $n$-manifolds whose sectional
  curvatures are all negative, and whose simplicial
  volume is bounded above by $v$.  Is $\mathcal{M}_{n,v}$ finite
  whenever $n\geq 4$?
\end{question}

Some comments, fixing $n\geq 4$ and $v>0$:
\begin{itemize}
\item The set $\mathcal{M}_{n,v}$ contains only finitely many constant
  curvature manifolds, even up to diffeomorphism.  This follows from Wang's finiteness
  theorem \cite{W} together with the proportionality 
  of simplicial volume to hyperbolic volume for closed manifolds (see, e.g., \cite[C.4]{BP92}).
\item  Fix $b>0$.  The set $\mathcal{M}_{n,v}$ contains only finitely
  many  diffeomorphism classes of manifolds with sectional curvatures
  $K$ pinched between $-1$ and $b$.
  Indeed, as noted in the inequality \eqref{propclosed}, the
  curvature assumption $-1\le K \le b<0$ 
  together with an upper bound on the simplicial volume gives
  an upper bound on the volume \cite[0.3, 1.2]{gromov:vbc}.
  On the other hand, 
  the condition $-1 \le K <0$ together with 
  the upper bound on the volume gives an upper bound on the diameter
  \cite[1.2]{G} and a lower bound on the volume \cite[1.3 A]{G}.
  Now, Cheeger's finiteness theorem applies 
  \cite[1.4]{G}.

  Thus if there are
  infinitely many diffeomorphism types of closed $n$--manifolds of
  bounded volume, with sectional curvatures $\le -1$,
  there cannot be any uniform lower bound on
  their sectional curvatures, so long as $n\ge 4$.  So a negative
  answer to Question \ref{q:manif} would have to involve manifolds
  without uniformly pinched curvature.
  For $n>3$ there do exist sequences of negatively curved
  manifolds whose curvatures cannot be
  uniformly pinched.  Gromov--Thurston \cite[0.5]{GT} describe a
  sequence $\{V_i\}$ of negatively curved
  closed manifolds which are $i$-fold ramified covers of a fixed
  closed manifold $V$ of constant negative curvature, and show the
  curvatures of the $V_i$
  cannot be uniformly pinched.
  We remark that such a sequence of ramified covers $V_i$ must satisfy
  $0<i\|V\| \le \|V_i\|$, so it 
  does not imply a negative answer to Question \ref{q:manif}.

\item 
  The sets $\mathcal{M}_{3,v}$ are infinite for large enough $v$, even if one
  restricts to constant curvature.  This is because of Thurston's
  hyperbolic Dehn surgery theorem, and Thurston's theorem (which we
  generalize in this note) that volume decreases under hyperbolic Dehn
  filling. 
\end{itemize}

It is easy to extend the definition 
of simplicial volume to closed
pseudomanifolds.  This is explained in Section \ref{s:defs} (cf. \cite{Ya97}).
Pseudomanifold versions of the bounds in
\eqref{propclosed}
have been established by Yamaguchi
\cite[Theorem 0.2 and 0.5]{Ya97}
as follows.  
First, if $X$ is an orientable $n$-dimensional, compact, geodesically
complete (locally) CAT($-1$) pseudo-manifold, then 
\begin{equation}\label{yamaguchilower}
\Vol(X)\leq  \frac{\pi}{(n-1)!} \|X\|.
\end{equation}
Yamaguchi conjectures that the factor of $\frac{\pi}{(n-1)!}$ in
\eqref{yamaguchilower} can be
replaced by the volume of a regular ideal hyperbolic $n$--simplex.
Second, if $X$ is an $n$-dimensional compact orientable
geodesic space without boundary such that
$K \ge -1$ in the sense of Alexandrov
and $X$ is ``locally Lipschitz contractible'', then 
  \[\|X\|\leq n!(n-1)^n\Vol(X),\]
where $\Vol(X)$ is the $n$-dimensional 
Hausdorff measure of $X$.  Both bounds cannot be applied at once
except in the manifold setting, since upper and lower curvature bounds
can only coexist in a manifold.  We do not use either bound in this
paper, except in Remark \ref{yamaguchiapp} below.

\medskip
We do not 
answer Question \ref{q:manif}, but we answer the analogous question
for negatively curved pseudomanifolds:
\begin{question}\label{q:pseud}
  Let $n\geq 4$.
  Fix $v>0$, and let $\mathcal{P}_{n,v}$ be the set
  of (homotopy classes of) closed locally CAT$(-1)$ $n$--pseudomanifolds whose simplicial
  volume is bounded above by $v$.  Is $\mathcal{P}_{n,v}$ finite?
\end{question}
Specifically, we apply a ``Dehn filling'' construction to give a
negative  answer to Question \ref{q:pseud}.

Let $M$ be a compact $n$--manifold with boundary a union of $m$
tori, whose
interior $V = \mathrm{Int}(M)$ admits a complete hyperbolic metric of finite volume.
Let $\overline{M}\cong M$ be a compact manifold
obtained from $V$
by removing horospherical neighborhoods of the cusps.  The
manifold $\overline{M}$ has $m$ boundary components $N_1,\ldots,N_m$,
each of which inherits a flat metric from the hyperbolic metric on $\mathrm{Int}(M)$.
In \cite{FM1}, we defined and studied \emph{fillings} of $M$,
which are topological spaces obtained by gluing the product of a torus
with the cone on another torus to each boundary component of
$\overline{M}$.  Here is an equivalent definition.
\begin{definition}\label{d:filling}
  For each $i\in \{1,\ldots,m\}$, let $T_i<N_i$ be a
  $k_i$--dimensional submanifold
  which is totally geodesic in $N_i$.  The
  $(n-1)$--torus $N_i$ is foliated by parallel copies of $T_i$, with
  leaf space $V_i$ homeomorphic to an $(n-1-k_i)$--dimensional torus.
  Let $h_i\co N_i\to V_i$ be the 
  quotient map, and define $M(T_1,\ldots,T_m)$ to be the quotient
  $\overline{M}/\sim$, where $x\sim x'$ if $h_i(x)=h_i(x')$ for some
  $i$.  The space $M(T_1,\ldots,T_m)$ is called a \emph{filling of
    $M$}. If all of the tori $T_i$ have
  injectivity radius strictly larger than $\pi$,
  $M(T_1,\ldots,T_m)$ is called a \emph{$2\pi$--filling of $M$}.
\end{definition}
In case $M$ is $3$--dimensional, and each $T_i$ is one-dimensional, 
fillings $M(T_1,\ldots,T_m)$ are the same as ordinary Dehn fillings of
$M$.  In \cite{FM1}, we generalize work in \cite{BlHo,Sc,MS} (including the
Gromov-Thurston $2\pi$ Theorem) to prove:
\begin{theorem}\label{t:filling}\cite{FM1}
Let $M$ be an $n$--dimensional hyperbolic manifold as above.
Every $2\pi$--filling $M(T_1,\ldots,T_m)$ of $M$ admits a locally
CAT$(0)$ metric, so the universal cover is CAT$(0)$ with isolated
flats. 
If every $T_i$ has dimension at least $n-2$, then 
$M(T_1,\ldots,T_m)$ admits a locally CAT$(-1)$ metric.
\end{theorem}

Each filling of $M$ is a closed pseudomanifold with singular set equal
to a disjoint union of tori of dimensions $n-(1+\dim(T_1)),\ldots
n-(1+\dim(T_m))$ (see Definition \ref{d:pseudomanifold} for these
terms).
In particular, its simplicial volume is well-defined \cite{Ya97}.
The main result of the current paper is that these fillings have 
positive simplicial volume, bounded above
by the simplicial volume of $M$.
\begin{theorem}\label{t:volumedecreases}
Let $M$ be a compact $n$--manifold with boundary a union of $(n-1)$--tori, and
suppose that $V=\mathrm{Int}(M)$ admits a complete hyperbolic metric of
finite volume $\Vol(V)$.
Let $M(T_1,\ldots,T_m)$ be a $2\pi$--filling of $\mathrm{Int}(M)$.
Then
\[ 0< \|M(T_1,\ldots,T_m)\| \leq \frac{\Vol(V)}{v_n}, \]
where $v_n$ is the volume of a regular ideal hyperbolic $n$-simplex.
\end{theorem}

\begin{corollary}
  The answer to Question \ref{q:pseud} is ``no,'' for $v$ equal to the
  volume of any hyperbolic $n$--manifold with torus cusps.
\end{corollary}
\begin{proof}
  Let $M$, $\overline{M}$, $N_1,\ldots,N_m$ be as in the discussion
  preceding Definition \ref{d:filling}.  For each $j\in
  \{1,\ldots,m\}$, choose a sequence $\{T_i^j\}_{i\in\mathbb{N}}$ of
  totally geodesic codimension $1$ tori in $N_j$, all with $\mathrm{inj}>\pi$, and
  so that
  $\lim_{i\to\infty}\mathrm{inj}(T_i^j) = \infty$.  Let $M_i$ be the
  filling $M(T_i^1,\ldots,T_i^m)$.  Proposition 2.12 of \cite{FM1}
  implies that for any $i$ only finitely many $M_j$ have fundamental
  group isomorphic to $\pi_1M_i$, so the set $\{M_i\}_{i\in
    \mathbb{N}}$ contains infinitely many homotopy classes.  Theorem
  \ref{t:filling} implies that $0<\|M_i\|\leq \frac{\Vol(V)}{v_n}$
  for all $i$.
\end{proof}

\begin{remark}
  In \cite{RT05}, Ratcliffe and Tschantz describe some nonpositively
  curved $4$--di\-men\-sion\-al homology spheres.  These homology
  spheres are obtained by $2\pi$--filling of a cusped hyperbolic
  manifold described by Ivan\v{s}i\'{c} in \cite{Iv}.  They seem to be
  the first exmples of $4$--dimensional homology spheres with
  Riemannian metrics of non-positive sectional curvature (see the
  paragraph after Remark 5 in \cite{Na-We}).  Belegradek in
  \cite{Be06} asked if these homology spheres have positive simplicial
  volume.  Our main theorem answers this question affirmatively, but
  the main part of our argument is to do with the upper bound on
  simplicial volume.  The reader primarily interested in positivity of
  simplicial volume can stop with Subsection \ref{ss:norms}.  In
  particular, we answer Belegradek's question with Corollary
  \ref{c:specialcase} at the end of that subsection.
\end{remark}

It would be interesting to get some further information on the
possible volumes of fillings of $M$.  Hyperbolic Dehn filling of a
$3$--manifold strictly decreases volume, and the volume difference
goes to zero as the length of the filling slope goes to infinity
(see \cite[6.5.6]{Th} for strictness
and \cite{NZ85,HoKe,FKP08} for estimates of the difference).  We conjecture that similar phenomena hold in higher
dimensions.

\begin{conjecture}
The volumes $\|M(T_1,\ldots,T_m)\|$
accumulate on $\frac{\Vol(V)}{v_n}$ but do not attain it.    
\end{conjecture}
\begin{question}
   Let
  $\epsilon>0$.  Is there some $R$ so that if every torus $T_i$
   has injectivity radius bigger than $R$, then
\[ \frac{\Vol(V)}{v_n}-\epsilon<\|M(T_1,\ldots,T_m)\|<\frac{\Vol(V)}{v_n}? \]
\end{question}
\begin{question}
  Consider the set of fundamental groups of $2\pi$--fillings of $M$ as a
  subspace of the space of ``marked groups'' generated by some fixed
  set of generators of $\pi_1M$.  Is the simplicial volume a
  continuous function on this space?
\end{question}
\begin{remark}\label{yamaguchiapp}
  It is possible to analyze the metrics constructed in \cite{FM1} and
apply Yamaguchi's inequality \eqref{yamaguchilower} to deduce that
the CAT$(-1)$ fillings of a fixed $M$ have simplicial volume bounded
away from zero, but this is about all we know for certain.
\end{remark}

\subsection{Outline}
In Subsection \ref{ss:norms} we recall the definition of simplicial volume
for manifolds and pseudomanifolds, and apply work of Mineyev--Yaman to
obtain the lower bound in the main theorem.  Subsection \ref{ss:symm}
and Appendix \ref{a:symm} are technical sections concerned with replacing
singular chains by ``symmetric'' singular chains.  The upper bound of
the main theorem is proved in Section \ref{s:main}.  In Appendix
\ref{s:prop} we give a proof of the proportionality of volume and
relative simplicial volume for cusped hyperbolic manifolds.

\section{Simplicial volume and duality}\label{s:defs}
In this section we recall the definition of
\emph{simplicial volume} for pseudomanifolds and set up some
computational tools. 

\subsection{Gromov norms}\label{ss:norms}
Let $X$ be any space, and let $\omega \in H_*(X;\bR)$ or $\omega \in
H_*(X,Y;\bR)$ for some $Y\subset X$.
The \emph{Gromov norm of $\omega$} is the smallest $l^1$--norm of a
real singular chain representing $\omega$,
\[ \|\omega\| = \inf\left\{\sum_{i=1}^k
|\alpha_i|\ \right|\ \left. \left[\sum_{i=1}^k\alpha_i\sigma_i \right]
 = \omega \right\}.\]
(Since this infimum may be zero, the ``Gromov norm'' is only a
 seminorm in general.)
This quantity only depends on $\omega$ and the homotopy type of the
pair $(X,Y)$.  In the absolute case,
we can use it to define the Gromov norm of any $\omega\in H_*(G)$,
where $G$ is a discrete group, by choosing $X$ to be a $K(G,1)$.

If $\omega$ is an integral cycle and $\eta$ is the canonical map from
integral to real homology, then we set
$\|\omega\| = \|\eta(\omega)\|$.   

The Gromov norm of a cycle
can be computed via the pairing between bounded
cohomology and $l^1$ homology:
If $\omega\in H_k(X)$ and $H_b^k(X;\bR)$ is the
$k$--dimensional bounded
cohomology of $X$, then
\begin{equation}\label{duality}
 \|\omega\| = \sup\left(\left.\left\{\frac{1}{\|\phi\|_\infty}\ \right| \phi\in
H_b^k(X;\bR),\langle \phi,\omega\rangle = 1\right\}\cup\{0\}\right).
\end{equation}
See \cite[Proposition F.2.2]{BP92} for a proof.  

For $M$ a closed orientable $n$--manifold,
Gromov defined the \emph{simplicial volume} $\|M\|$ to be the Gromov
norm of the fundamental class in $H_n(M;\bR)$.
For a
manifold with boundary, it is natural to define 
$\|M,\partial M\|$ to be
the Gromov norm of the fundamental class in $H_n(M,\partial M;\bR)$.  
It is clear from the definition that this volume only depends on the
homotopy type of $(M,\partial M)$.  For closed hyperbolic manifolds
the simplicial volume is
proportional to the volume.  More generally, if $M$ is compact, and the interior
of $M$ admits a complete finite volume hyperbolic metric of
volume $\Vol$, then
\[ \|M,\partial M\| = \frac{1}{v_n} \Vol,\]
where $v_n$ is the volume of a hyperbolic regular ideal simplex of
dimension $n$. (See \cite[C.4]{BP92} for the closed case; we show how
to extend to the finite volume case in Appendix \ref{s:prop}.)

In \cite{Ya97}, Yamaguchi extended the notion of simplicial volume to
orientable pseudomanifolds, defined as follows.
\begin{definition}\cite[Definition 2.1]{Ya97}\label{d:pseudomanifold}
  A locally compact, Hausdorff, locally homologically $n$--connected
  space $X$ is called an   \emph{oriented $n$--pseudomanifold} if it
  contains a closed subset $S$ (the \emph{singular set}) so that:
  \begin{enumerate}
    \renewcommand{\labelenumi}{(\roman{enumi})}
  \item $X\smallsetminus S$ is an orientable $n$--manifold,
  \item $\dim(S)\leq n-1$, and
  \item $\dim(S\cap\overline{X\smallsetminus S}) \leq n-2$.
  \end{enumerate}
\end{definition}
\begin{remark}
In this definition, we allow the possibility that $S$ is not ``really''
singular.  For example, $X$ could be a manifold and $S$ a
codimension $2$ submanifold.  This is important in applying
Proposition \ref{p:nonzero} to $2\pi$--fillings which happen to be
manifolds, as in Corollary \ref{c:specialcase} below.
\end{remark}
In \cite[Proposition 2.2]{Ya97}, Yamaguchi shows that an oriented $n$--pseu\-do\-man\-i\-fold $X$
has a fundamental class $[X]$ in $H_n(X) = H_n(\pi_1(X))$ so that the
orientation cocycle $\beta \in H^n(X)$ is dual to this class.
For such a pseudomanifold $X$, the \emph{simplicial volume} of $X$ is
defined to be the Gromov norm of $[X]$, and written $\|X\|$.

If $X$ is compact with CAT$(-1)$ universal cover, then $X$ is aspherical and
$\pi_1(X)$ is (word) hyperbolic \cite{gromov:hg}.  In this case
the
orientation cocycle is cohomologous to a bounded cocycle by the main
theorem of \cite{min:bcchg}, and so $\|X\|>0$ by the duality equation
\eqref{duality}.  

If $X$ is compact and the universal cover is CAT$(0)$ with isolated
flats, then $X$ is again aspherical, but $\pi_1(X)$ is only relatively
hyperbolic, relative to fundamental groups of images of flats in $X$ \cite{HK}. 
In this case, we will need the
following corollary of a theorem of Mineyev--Yaman:
\begin{proposition}\label{p:nonzero}
  Let $M$ be a closed orientable aspherical
  $n$--dimensional pseudomanifold with
  singular set $V\subset M$ which is a disjoint union of aspherical components $V
  = \cup_{i=1}^mV_i$.  Suppose that $\pi_1(M)$ is hyperbolic relative
  to the collection of subgroups
  $\left\{\pi_1(V_i)\mid i\in \{1,\ldots,n\}\right\}$.  Then $\|M\|>0$.
\end{proposition}
\begin{proof}
  The inclusion of pairs $(M,\emptyset)\to (M,V)$ gives a natural map 
  \[ H^n(M,V;\bR)\to H^n(M;\bR). \]
  Since $\dim(V\cap \overline{M\smallsetminus V})\leq n-2$, the
  Mayer--Vietoris exact sequence shows this map to be an isomorphism
  (cf. \cite[Proposition 2.2]{Ya97}).  The main result of
  Mineyev--Yaman \cite{MY} implies that the map 
  \[ H_b^n(\pi_1(M),\{\pi_1(V_i)\}_{i=1}^m;\bR) \to
    H^n(\pi_1(M),\{\pi_1(V_i)\}_{i=1}^m;\bR)\]
  is surjective, since $\pi_1(M)$ is hyperbolic relative to the
  fundamental groups of the components of $V$.  Since $M$ and the
  components of $V$ are aspherical, 
  this is the same as saying that the map
  \begin{equation}\label{allbounded}
    H^n_b(M,V;\bR) \to H^n(M,V;\bR) = H^n(M;\bR)
  \end{equation}
  is surjective.   It follows that the
  orientation cocycle of $M$ is cohomologous to a bounded cocycle $\alpha$.
  (The cocycle moreover can be chosen to vanish on chains in $V$, but we don't use
  this.)  Equation \eqref{duality} gives
  $ \|M\| \geq \frac{1}{\| \alpha\|_\infty} > 0 .$
\end{proof}
Applying Theorem \ref{t:filling},
we immediately obtain the lower bound of Theorem \ref{t:volumedecreases}.
\begin{corollary}\label{c:lowerbound}
  If $M(T_1,\ldots,T_m)$ is a $2\pi$--filling as in Theorem
  \ref{t:volumedecreases}, then 
  \[\|M(T_1,\ldots,T_m)\|>0.\]
\end{corollary}
The special case of nonpositively curved
\emph{manifolds} obtained by $2\pi$--filling of $\geq 4$--dimensional
hyperbolic manifolds has been studied before us
by Schroeder, Anderson and others
\cite{Sc,An}.  
Ivan\v{s}i\'{c} describes in \cite{Iv} a particular cusped hyperbolic
manifold which is the complement of five linked $2$--tori in the $4$--sphere.  Ratcliffe
and Tschantz point out in \cite{RT05} that appropriately chosen
$2\pi$--fillings of this hyperbolic manifold are homology
$4$--spheres.  
In
\cite[Question 7.2]{Be06}, Belegradek asked if such fillings have
positive simplicial volume.  
As a
special case of \ref{c:lowerbound} above, we give an affirmative answer.
\begin{corollary}\label{c:specialcase}
  Every manifold obtained by
  $2\pi$--filling of an orientable hyperbolic $n$--manifold with torus cusps
  has positive simplicial volume.  In particular,
  the aspherical homology $4$--spheres constructed by Ratcliffe and
  Tschantz have positive simplicial volume.  
\end{corollary}
Theorem \ref{t:volumedecreases} gives further information (an upper
bound) on the simplicial volume. 

\subsection{Symmetrization}\label{ss:symm}
It will be convenient in the arguments of Section \ref{s:main} to
assume that all chains considered are \emph{symmetric}.  Here we say
what this means and justify the assumption.

Consider the abstract $n$--simplex $\Delta^n$.  The symmetric group
$S_{n+1}$ acts in an obvious way on the vertices, and any permutation
$p\in S_{n+1}$ extends to $\Delta^n$ as an affine transformation.

Let $X$ be some space.
Any singular simplex $\sigma\co \Delta^n\to X$ can be altered by
precomposition with some $p\in S_{n+1}$ and the singular simplices
$\sigma$ and $\sigma\circ p$ are generally different, though if $p$ is
an odd permutation and $q$ is an even one
(so $\mathrm{sgn}(p)=-1$ and $\mathrm{sgn}(q)=1$), we'd really like to
think of $\sigma\circ p$ as being the same as $\sigma \circ q$ or $\sigma$, just
with opposite orientation.  We would also like to avoid the annoying
bookkeeping which is standard when dealing with simplices.
We therefore define the \emph{symmetrization} map
\[ S\co C_*(X;\bR)\to C_*(X;\bR) \]
on singular $n$-simplices $\sigma$ by
\begin{equation}\label{symmsum}
S(\sigma) = \frac{1}{(n+1)!}
   \sum_{p\in S_{n+1}}\mathrm{sgn}(p)\left(\sigma\circ p\right),
\end{equation}
and extend linearly over $C_*(X;\bR)$.  
\begin{lemma}\label{l:chhom}
  $S$ is a chain map, chain homotopic
  to the identity.  Moreover
  $\|S(c)\|_1\leq\|c\|_1$ for any chain $c$.
\end{lemma}
\begin{proof}
  It is evident that norm does not increase.  The chain homotopy is
  described in Appendix \ref{a:symm}.  
\end{proof}
\newcommand{\symch}{C_*^S}
Define 
$\symch(X;R)$ to be the subcomplex of $C_*(X;R)$ consisting of chains
in the image of $S$;  the transformation $S$ restricts to the identity
on this subcomplex.
Lemma \ref{l:chhom} implies that in computing Gromov norms, we may as
well restrict attention to $\symch(X;R)$.

Suppose $X$ has the structure of a 
simplicial complex (or more generally just a CW complex so that each
cell is a simplex, and so that each restriction of a gluing map to a face is affine)
with triangulation $\mc{T}$.
\begin{definition}
We call a singular simplex $\sigma\co\Delta^n \to X$ \emph{affine} if
$\sigma = \tau\circ a$ where
$a\co \Delta^n\to \Delta^k$ is some affine surjection taking vertices
to vertices, and
$\tau$ is 
the characteristic map of some $k$--simplex of $\mathcal{T}$.  A
linear combination of affine simplices is an \emph{affine chain}, and
so on.
\end{definition}

\begin{remark}\label{r:inclusion}
There is a nice inclusion 
\newcommand{\simpch}{\mc{C}^{\mc{T}}_*}
\[ \simpch(X;\bR) \stackrel{i}{\longrightarrow} \symch(X;\bR) \]
of simplicial chains $\simpch(X;\bR)$ given by sending a simplex 
of $X$ to the
symmetrization of its characteristic map.  
This gives a way to think of simplicial chains as
symmetric affine singular chains.  
\end{remark}

In fact, the image of the map $i$ from Remark \ref{r:inclusion}
contains \emph{all} of the symmetric affine simplices;  the symmetrization of
an affine simplex which drops dimension must vanish:
\begin{lemma}\label{l:vanish}
  Let $X$ have a triangulation $\mathcal{T}$, and suppose that
  $\sigma$ is an affine singular $n$--simplex with image in the
  $(n-1)$--skeleton of $X$.  Then $S(\sigma) = 0$.
\end{lemma}
\begin{proof}
  Since $\sigma$ has image in the $(n-1)$--skeleton, there must be a
  pair of vertices $\{v_i,v_j\}$ of the standard simplex $\Delta^n$
  which are identified by $\sigma$.  Let $p$ be a permutation of the
  vertices of $\Delta^n$, and let $\tau$ be the transposition switching
  $v_i$ and $v_j$.  The maps $\sigma\circ p$ and $\sigma\circ\tau\circ
  p$ are the same as singular simplices, but the permutations $p$ and
  $\tau\circ p$ have opposite signs, so the terms
  $\mathrm{sgn}(p)(\sigma\circ p)$ and $\mathrm{sgn(\tau\circ
    p)}(\sigma\circ\tau p)$ cancel in the sum \eqref{symmsum}.
\end{proof}

\section{Main result}\label{s:main}
In this section we prove the main result, Theorem
\ref{t:volumedecreases}.  The strategy is to modify an efficient relative
fundamental cycle for $(M,\partial M)$ to a fundamental cycle on the
filling, without increasing norm by too much.
\begin{lemma}\label{l:rational}
  Let $X$ be a topological space.
  Let $c\in {C_n}(X;\bR)$ be homologous to a rational
  chain, and let $\epsilon>0$.
  Then there is some $c'\in
  C_n(X;\bQ)$ homologous to $c$ so that $\|c'-c\|_1<\epsilon$.  If $c$
  is symmetric, $c'$ can be chosen to be symmetric.
\end{lemma}

\begin{proof}
  By assumption, there is some $f\in C_{n+1}(X;\bR)$ so that
  $c+\partial f \in C_n(X;\bQ)$.  Since $C_{n+1}(X;\bQ)$ is dense in
  $C_{n+1}(X;\bR)$, there is some $f'\in C_{n+1}(X;\bQ)$ with
  $\|f-f'\|<\frac{\epsilon}{n+2}$.  The chain $c'=c+\partial(f-f')$
  satisfies $\|c'-c\|<\epsilon$, since $\|\partial_{n+1}\|=n+2$.

  If $c$ is symmetric, we can take $f$ and $f'$ to be symmetric in the
  above argument, and it will follow that $c'$ is symmetric.
\end{proof}

\begin{definition}\label{d:small}
  Let $\mathcal{T}$ be a finite triangulation of a space $X$, and let
  $\mathcal{T}'$ be its barycentric subdivision.
  The \emph{(closed) star neighborhood} of a simplex $\tau$ of $\mc{T}$
  is the union of those simplices of the subdivision $\mathcal{T}'$
  intersecting $\tau$.  A point in the interior of an $n$--simplex of
  $X$ lies in at least one closed star neighborhood of a vertex, and
  in at most $n+1$ such closed stars.

  The \emph{open star neighborhood} of a simplex is
  the interior of its star neighborhood.  (So the open star
  neighborhood of a simplex does not include its proper faces.)  The
  open star neighborhoods of the simplices of $\mc{T}$ form an open
  cover of $X$.

  Let
  $\sigma\co \Delta^n\to X$ be a singular simplex.  We say $\sigma$ is
  \emph{small} with respect to $\mathcal{T}$
  if the image of $\sigma$ is contained entirely in the open star of
  some vertex of $\mathcal{T}'$.

  A singular chain will be called \emph{small} if all its simplices
  are small.
\end{definition}
Given an order on the vertices of $\mathcal{T}$ there is a canonical
way to homotope a small chain to an affine chain, as we describe in
the next lemma.

\begin{lemma}\label{l:affine}
  Let $X$ be a space with a (combinatorial) triangulation $\mathcal{T}$, and let
  $<_\mathcal{T}$ be a total order on the vertices of $\mathcal{T}$.
  Canonically associated to any small singular simplex $\sigma\co \Delta^n\to X$
  are
  \begin{enumerate}
  \item a ``straightened'' affine simplex $a(\sigma)$, and
  \item a homotopy $f_t(\sigma)$ so that $f_0(\sigma)=\sigma$ and
    $f_1(\sigma)=a(\sigma)$, 
  \end{enumerate}
  so that if $F$ is a face of $\Delta^n$, then 
  $a(\sigma|F)=a(\sigma)|F$ and $f_t(\sigma|F)=f_t(\sigma)|F$.  (In
  other words, homotopies on a face depend only on that face, so a
  cycle remains a cycle throughout the homotopy.)  
  Moreover, if
  $\sigma$ is invariant under precomposition with an affine map
  $\tau$, so are $a(\sigma)$ and  $f_t(\sigma)$.
\end{lemma}
\begin{proof}
  The idea is illustrated quite well by the figure given in the proof of
  Proposition 6.5.1 in Thurston's notes \cite{Th}.  
  We give a little more detail.  (The figure given
  there is not of a combinatorial triangulation.  In that case a total
  order on the ``corners'' of the simplices must be given, rather than
  on the vertices.  We avoid this complication here by assuming a
  combinatorial triangulation.)

  We begin by defining the straightening $a(\sigma)$.
  For $x\in X$ define $n(x)$ to be the least vertex $v$
  of $\mathcal{T}$ so that $x$ is in the closed star neighborhood of
  $v$.  Let $\{v_0,\ldots,v_n\}$ be the vertices of the standard
  simplex $\Delta^n$.
  Since $\sigma$ is small, the image of $\sigma$ lies in the
  open star neighborhood of some simplex $\tau$ of $\mathcal{T}$.  (In fact
  there may be several choices for $\tau$, but the maps $a(\sigma)$
  and $f_t(\sigma)$ we define will not depend on the choice.)
  It follows
  that $n(\sigma(v_i))$ is a vertex of $\tau$ for each $i$.  We define
  $a(\sigma)(v_i) = n(\sigma(v_i))$ and then extend to an affine map
  $a(\sigma)\co \Delta^n\to \tau$.  It is clear that if
  $F$ is a face of $\Delta^n$ we have $a(\sigma|F) = a(\sigma)|F$.

  We now must define the homotopies $f_t(\sigma)$.  For each $x\in
  \Delta^n$, we let $\tau(x)$ be the unique simplex of $\mathcal{T}$
  so that 
  \begin{enumerate}
  \item\label{containsx} $\sigma(x)\in \tau(x)$, and
  \item $\tau(x)$ has minimal dimension subject to \eqref{containsx}.
  \end{enumerate}
  Since $\tau(x)$ lies in the open star neighborhood of $\tau$, the
  simplex $\tau$ is a face of $\tau(x)$.  Thus both
  $\sigma(x)$ and $a(\sigma)(x)$ lie in $\tau(x)$.  Convex combination is
  well-defined for points in the same simplex, so we may take, for
  $t\in [0,1]$,
  \begin{equation*}
  f_t(\sigma)(x) = (1-t)\sigma(x) + t\, a(\sigma)(x). 
  \end{equation*}
  The reader may check that $f_t$
  is well-defined, continuous, and satisfies the property
  $f_t(\sigma|F)=f_t(\sigma)|F$ for each face $F$ of $\Delta^n$.

  Finally, we note that if $\tau$ is some affine map from $\Delta^n$
  to itself so that $\sigma\circ \tau = \sigma$, then
  $a(\sigma\circ\tau) = a(\sigma)$ and therefore $f_t(\sigma\circ\tau)
  = f_t(\sigma)$ for all $t$.
\end{proof}

\begin{lemma}\label{l:simplicial}
  Let $N$ be a manifold, and let
  $\mathcal{T}$ be a triangulation of $\partial N$.  
  If $c\in {C_k}(N; \bR)$ is a relative singular cycle in
  $(N,\partial N)$ so that  $\partial c$ is small with respect to 
  $\mathcal{T}$, then $c$ is (relatively)
  homologous to
  $c'$ with $\|c'\|_1\leq \|c\|_1$ and $\partial c'$ an affine chain.  If
  $c$ was symmetric, then so is $c'$.
\end{lemma}
\begin{proof}
  The chain $c$ is a finite linear combination 
  $c = \sum_{i=1}^{m'}\lambda_i\sigma_i$ of singular simplices 
  $\sigma_i\co \Delta^k\to N$.  The idea is to ``preassemble'' the simplices
  to get a map $\phi\co K_c \to N$ from a certain CW complex $K_c$, so
  that $c$ is  $\phi_\sharp$ of a linear combination of characteristic
  maps.   We can use Lemma \ref{l:affine} to homotope the ``boundary''
  of $K_c$ to a simplicial map, and then apply the homotopy extension
  property of CW pairs.

  More precisely, let $J = \{1,\ldots m'\}/\sim$, where $i\sim j$ if
  $\sigma_i = \sigma_j\circ q$ for some permutation $q$.  We define
  \begin{equation}\label{kcdef}
    K_c = \left( \coprod_{[i]\in J} \Delta^k_{[i]} \right)/\sim 
  \end{equation}
  where $x_1\sim x_2$ whenever there are $(k-1)$--dimensional faces
  $F_1\subset \Delta^k_{[i_1]}$ and $F_2\subset \Delta^k_{[i_2]}$
  satisfying
  \begin{enumerate}
  \item $x_j\in F_i$ for $j\in \{1,2\}$, and
  \item there is an affine map $\tau\co F_1\to F_2$ so that $\tau(x_1)
    = x_2$ and $\sigma_{i_1}|F_1 = \sigma_{i_2}\circ\tau$.
  \end{enumerate}
  That is, if two singular simplices agree on a codimension $1$ face, we glue
  the corresponding faces together in $K_c$.  Moreover, if a
  codimension $1$ face is unchanged by precomposing with some isometry,
  we identify the face to itself by the corresponding
  symmetry.  
  We thus obtain a CW complex $K_c$ and a map $\phi\co
  K_c\to N$ given by $\phi(x) = \sigma_{i}(x)$ where $x\in
  \Delta^k_{[i]}$.  In the complex $K_c$
  we can canonically build a chain $\tilde{c}$ made up of
  ``characteristic maps'' $\Delta^k_{[i]}\to K_c$ so that
  \[ c = \phi_\sharp(\tilde{c}).\]
  If $c$ was symmetric, then so is $\tilde{c}$.

  The boundary $\partial\tilde{c}$ is supported in some minimal
  subcomplex of $K_c$, which we'll refer to as $\partial K_c$.  
  Because $c$ is a relative cycle, the
  map $\phi$ described in the last paragraph is really a map of pairs
  \[ \phi \co (K_c,\partial K_c)\to (N,\partial N) .\]
  
  After fixing some total order on the simplices of $\mathcal{T}$,
  Lemma \ref{l:affine} gives a homotopy in $\partial N$ of
  $\phi|\partial K_c$ to a ``simplicial'' map $\psi$; this map is
  simplicial in the sense that if $F$ is any face of one of the
  simplices $\Delta^k_{[i]}$ used to define $K_c$, and if $\iota_F$ is
  the inclusion map, then $\phi\circ \iota_F$ is affine, with image in a
  single simplex of $\mathcal{T}$.

  By the homotopy extension property of CW pairs, the homotopy from
  $\phi|\partial K_c$ to $\psi$ can be extended to all of $K_c$.  Let 
\[ \Phi_t \co K_c \to N\mbox{, }t\in[0,1] \]
  be this homotopy, with $\Phi_0 = \phi$.  The chain
  \[ c' = (\Phi_1)_\sharp(\tilde{c}) \]
  has all the properties we want.
\end{proof}

We need to show that tori can be triangulated so that every
simplex of diameter much smaller than the injectivity radius is
small.  Say that a triangulation of a torus
is \emph{$K$--fat} if every singular simplex
of diameter less than $K$ is small with respect to that triangulation.

\begin{lemma}\label{l:fattori}
  For each $d\geq 1$ there is a constant $0<K_d<1$ so that every flat torus of
  dimension $d$ with injectivity radius $R$ has a $(K_dR)$-fat
  triangulation with $t_d = 2^d d!$ simplices.
\end{lemma}
\begin{proof}
  Let $T$ be a flat $d$--torus, which we can think of as
  $\bR^d/\Lambda$, where $\Lambda$ is some rank $d$ lattice.  Every
  such lattice has a Lov\'asz--reduced basis \cite{LLL82} 
  $\mc{B} = \{\bv_1,\ldots,\bv_d\}$.
  For $k\in \{1,\ldots,d\}$ define $\theta_k$ to be the angle made by
  $v_k$ with the subspace of $\bR^d$ spanned by the rest of the basis.
  Babai \cite{Babai86} showed that, for all k,
  \[ \sin(\theta_k)\geq\left(\frac{\sqrt{2}}{3}\right)^d.\] 

  Any basis $\mc{B}$ of $\Lambda$ gives rise in a canonical way to a triangulation
  $\mc{T}_{\mc{B}}$ of $T$ with $t_d$ simplices
  given by the first barycentric subdivision of the parallelepiped
  spanned by $\mc{B}$.  

  Let $V$ be the set of all ordered bases of lattices in $\bR^d$, topologized
  as a subset of $\bR^{d^2}$, and
  let $F\co V\to \bR_+$ be the function which sends $\mc{B}$ to the
  least real $k$ so that $\mc{T}_{\mc{B}}$ is $k$--fat.  The number
  $F(\mc{B})$ is also the least number $k$ so that 
  triangulation of $\bR^d$ covering $T_{\mc{B}}$ is $k$--fat.
   We note the
  following useful properties of $F$:
  \begin{enumerate}
  \item $F$ is continuous and positive everywhere on $V$.
  \item\label{scale} If $\lambda\in \bR_+$ and $\mc{B}\in V$, then 
    $F(\lambda \mc{B}) = \lambda F(\mc{B})$.
  \item\label{squash} If there is a length-decreasing map from $\mc{B}$ to $\mc{B'}$
    then $F(\mc{B})\geq F(\mc{B}')$.
  \end{enumerate}

  Let $U\subset V$ be the set of bases consisting of unit vectors, any
  two of which make an angle $\theta$ with $\sin(\theta)\geq
  \left(\frac{\sqrt{2}}{3}\right)^d$.  Since $U$ is 
  compact, $F$ takes some minimum, positive value $L$ on $U$.  We
  will show that every torus of injectivity radius $\geq 1/2$ has a
  triangulation which is $L$--fat.  Linearity of $F$ (property
  \eqref{scale} above) then implies the lemma for $K_d = 2L$.

  Let $T=\bR^d/\Lambda$ be a torus of injectivity radius at least $1/2$.
  A Lov\'asz-reduced basis $\mc{B}=\{\bv_1,\ldots,\bv_d\}$ for $\Lambda$ has all angles
  between basis vectors within $[\pi/3,2\pi/3]$, by Babai's result.
  Since the injectivity radius of $T$ is at least $1/2$, there is a
  distance-decreasing  map from $\mc{B}$ to a basis in $U$, namely 
\[
 \mc{B}'=\left\{\frac{\bv_1}{\|\bv_1\|},\ldots,\frac{\bv_d}{\|\bv_d\|}\right\}. \]
  Property \eqref{squash} above tells us that $F(\mc{B})\geq
  F(\mc{B}')\geq L$, so the triangulation $\mc{T}_{\mc{B}}$ is
  $L$--fat, and the proof of the lemma is completed.
\end{proof}

\begin{lemma}\label{l:thurstonlemma}
For each $d\geq 1$ and each $s$ there is a constant $K(d,s)$ so that
any null-homologous
real singular $s$--cycle $z$ in a $d$--dimensional torus bounds a singular
$(s+1)$--chain $c$ with
\[ \|c\|_1\leq K(d,s) \|z\|_1\]
\end{lemma}
\begin{proof}
We follow the proof of \cite[Proposition 6.5.1]{Th},which discusses
the case $d=s=2$. 
Let $T$ be a $d$--dimensional torus, and let $z$ be an $s$--cycle.
Since there are only finitely many singular simplices in $z$,
their lifts to the universal cover $\bR^d$ have diameters
bounded by some $D>0$.  Let $K_d$ be the constant from Lemma
\ref{l:fattori}, and let $\tilde{T}\stackrel{\pi}{\longrightarrow} T$
be some cover of $T$ with injectivity radius bigger than $D/K_d$.  By
Lemma \ref{l:fattori}, there is a triangulation $\mathcal{T}$ 
of $\tilde{T}$ with $t_d$ simplices so that 
if $\tilde{\sigma}\co \Delta^s\to \tilde{T}$ is any lift of a singular
simplex appearing in $z$, then $\tilde{\sigma}$ is small with respect
to $\mathcal{T}$ (in the sense of Definition \ref{d:small}).   
The proof of Lemma \ref{l:fattori} shows that the combinatorics of
this triangulation are independent of the particular chain $z$ and the
cover $\tilde{T}$.

Write $C_*(X)$ for the real singular chain complex of $X$.
There is a transfer map  
\[ \mathrm{trans}\co C_*(T;\bR)\to C_*(\tilde{T};\bR)\]
which sends any singular simplex to the average of its lifts.  The map
$\mathrm{trans}$ is an isometric embedding with respect to the $l^1$
norm.  
The map
$\pi_\sharp\co C_*(\tilde{T};\bR)\to C_*(T;\bR)$ induced by the
covering is norm decreasing and 
$(\pi_\sharp\circ \mathrm{trans})$ is the identity map.

By the choice
of $\tilde{T}$ and $\mathcal{T}$, the chain $\mathrm{trans}(z)$ is
small.  By Lemma \ref{l:affine}, the simplices of 
$\mathrm{trans}(z)$ can be homotoped 
simultaneously and consistently to the affine singular simplices
making up the singular chain $a(\mathrm{trans}(z))$.
Triangulating this homotopy in a standard way, we find
that $\mathrm{trans}(z) - a(\mathrm{trans}(z))$ bounds an $(s+1)$--chain
$c_1$
in $T$, so that
\[ \|c_1\|_1 \leq C_1\|z\|_1 \]
for some constant $C_1$ depending on the dimension $s$ but otherwise
independent of $z$.  The space of affine real singular 
$s$--boundaries $B$ is
finite dimensional, and only depends on the combinatorics of the triangulation
$\mathcal{T}$.  

The \emph{affine filling norm} $\|b\|_{\mathrm{fill}}$ on $B$ is defined,
for a boundary $b\in B$, to be the smallest $l^1$--norm of an affine
$(s+1)$--chain bounded by $b$.   Let
\[C_2(d,s) = \sup \left\{\|b\|_{\mathrm{fill}}\mid b\in B \mbox{ and } \|b\|_1=1\right\},\]
and note that $C_2(d,s)$ only depends on the
combinatorics of $\mathcal{T}$.

The boundary $a(\mathrm{trans}(z))$ bounds an $(s+1)$--chain
$c_2$
in $T$ of $l^1$--norm at most $C_2(d,s)\|a(\mathrm{trans}(z))\|_1\leq C_2\|z\|_1$.  
Thus we have 
\[ \mathrm{trans}(z) = \partial (c_1+c_2) \]
and 
\[ z = \pi_\sharp (\partial (c_1+c_2) ) = \partial (\pi_\sharp(c_1 + c_2))\]
 and 
\[ \|\pi_\sharp(c_1 + c_2)\|_1\leq \|c_1+c_2\|_1 \leq \left(C_1(s)+ C_2(d,s)\right)\|z\|_1.\]
\end{proof}

We now prove the main theorem.

\begin{proof}[Proof of Theorem \ref{t:volumedecreases}]
We suppose that $M$ is a compact $n$--manifold with boundary a union
of tori, so that the interior admits a complete hyperbolic metric of
finite volume.  We fix a $2\pi$--filling $M(T_1,\ldots,T_m)$.  
To prove the theorem we must show
\[0< \|M(T_1,\ldots,T_m)\| \leq \frac{\Vol(V)}{v_n}, \]
where $v_n$ is the volume of a regular ideal hyperbolic $n$--simplex.

The lower bound follows from \cite{MY} and is given above in Corollary
\ref{c:lowerbound}.

We now establish the upper bound.
The strategy is to first
choose a symmetric representative $c_0$ for the (relative)
fundamental class of $(M,\partial M)$
which is close to optimal, and then modify it to
obtain a fundamental class $c_0'$
for $M(T_1,\ldots,T_m)$.  If such a
modification can be done in such a way that 
$|c_0'|_1$ is within some arbitrarily small constant of
$\|M,\partial M\|$, then we will have shown: 
\[ \|M(T_1,\ldots,T_m)\|\leq \|M,\partial M\| \]

The first step is to pass to a rational cycle.  Let $\epsilon>0$.
Let $c_0\in C_n(M;\bR)$ be a real symmetric
singular relative cycle representing
$[M,\partial M]$, and satisfying 
\begin{equation}
\|c_0\|_1<\|M,\partial M\|+\epsilon
\end{equation}
By Lemma \ref{l:rational}, there is a symmetric
rational singular relative cycle
$c_1\in C_n(M;\bQ)$ homologous to $c_0$ with
$\|c_1-c_0\|_1<\epsilon$.

We now pick a finite cover 
and triangulation of the boundary so that $\partial c_1$ lifts to a
chain which is small with respect to that triangulation.  Let $t_{n-1}
= 2^{n-1}(n-1)!$.
\begin{claim}\label{claim:passtocover}
We can
choose a finite cover $\tilde{M}\stackrel{\pi}\longrightarrow M$
and a triangulation $\mc{T}$ of $\partial\tilde{M}$
satisfying the conditions: 
\begin{enumerate}
\item\label{degree} The restriction of $\pi$ to any component of $\partial\tilde{M}$ has degree
  bigger than $m/\epsilon$, where $m$ is the number of cusps of $M$.
\item\label{small} If $\mathrm{trans}\co C_*(M;\bR)\to
  C_*(\tilde{M};\bR)$ is the transfer map, then
  $\mathrm{trans}(\partial c_1)$ is small with respect to $\mc{T}$.
\item\label{boundedcomb} Each boundary component of $\tilde{M}$ is
  triangulated by $t_{n-1}$ $(n-1)$--simplices.
\end{enumerate}
\end{claim}
\begin{proof}
  The compact manifold $M$ is homeomorphic to $\overline{M}$ as
  described in the introduction,
  a hyperbolic manifold
  minus a union of disjoint horospherical 
  neighborhoods of its cusps.  We identify $M$ with $\overline{M}$.
  Any simplex in the support of $\partial c_1$ therefore lifts to $\bH^n$.  
  Only finitely many such simplices occur, so there some number $r$
  bounding the diameter of any such lift.  Residual finiteness of
  $\pi_1M$ implies
  that we can pass to a finite cover $\tilde{M}\to M$ so that 
  \begin{enumerate}
  \item  Every boundary
    component of $\tilde{M}$ has injectivity radius bigger than $r/K_{n-1}$,
    where $K_{n-1}$ is the constant from Lemma \ref{l:fattori}, and
  \item  Every boundary component $N$ of $\tilde{M}$ has area at least $m/\epsilon$ times
    more than the area of the component of $\partial M$ covered by $N$.
  \end{enumerate}
  Lemma \ref{l:fattori} implies that $\partial \tilde{M}$  admits an
  $r$--fat triangulation with $t_{n-1}$ $(n-1)$--simplices per component 
  so that singular simplices lifted from $\partial c$ are all small
  respect to this triangulation.  
\end{proof}

Let $D$ be the degree of the cover coming from Claim
\ref{claim:passtocover}.  
Since $c_1$ is rational and symmetric,
$\mathrm{trans}(c_1)$ must also be rational and symmetric.  In particular,
\[ \mathrm{trans}(c_1) = \frac{1}{q} c_2 \]
for some integral symmetric chain $c_2\in C_*^S(\tilde{M};\bZ)$ and some positive
integer $q$.  Homologically $[c_2] =
\frac{q}{D}[\tilde{M},\partial\tilde{M}]$ in
$H_{n}(\tilde{M},\partial\tilde{M};\bR)$.  Since $\mathrm{trans}\co
C_*(M;\bR)\to C_*(\tilde{M};\bR)$ is
isometric, $\|c_2\|_1 = q\|c_1\|_1$.

Since $\mathrm{trans}(\partial c_1)$ is small with respect to
$\mc{T}$, the chain $\partial c_2$ is also small with respect to
$\mc{T}$.
Using Lemma \ref{l:simplicial}, we find a symmetric
integral chain $c_3$ with $\|c_3\|_1\leq \|c_2\|_1$ so that $c_3$
is (relatively) homologous to $c_2$ and
so that $\partial c_3$ 
is an affine chain, i.e., $\partial c_3\in C^{\mc{T}}_{n-1}(\partial
\tilde{M};\bZ)$. 

Let $\tilde N_1,\ldots,\tilde N_p$ be the boundary components of $\tilde{M}$.
We have $[c_3] = \frac{q}{D}[\tilde{M},\partial\tilde{M}]$ in
$H_{n}(\tilde{M},\partial\tilde{M})$ and $[\partial c_3]
=\frac{q}{D}\sum_{i=1}^p[\tilde N_i]$ in $H_{n-1}(\partial\tilde M)$.

The cycle $\partial c_3$ is symmetric, affine and represents $\frac{q}{D}$
times the fundamental class of $\partial \tilde M$.  By Lemma
\ref{l:vanish}, 
$\partial c_3$ has no ``degenerate'' affine simplices
in its support.
It follows that
$\partial c_3$ consists of exactly $\frac{q}{D}$ appropriately oriented copies
of every $(n-1)$--simplex in $\mc{T}$.

Each of the $p$
boundary components of $\tilde{M}$ is triangulated with exactly
$t_{n-1}$ simplices of dimension $n-1$.
We therefore have
\[\|\partial c_3\|_1 = p\, t_{n-1} \frac{q}{D}.\]
For each $i$ between $1$ and $n$, the $i$th boundary component of $M$
is covered by $p_i$ boundary components of $\tilde{M}$, so that
$p=\sum_i p_i$.  The degrees of these components are
$d_{i1},\ldots,d_{ip_i}$, and $\sum_j p_{ij} = D$ for each $i$.  Since
each $d_{ij}\geq \frac{m}{\epsilon}$, we get
\[ D\geq p_i \frac{m}{\epsilon} \forall i, \]
which implies that $mD \geq p \frac{m}{\epsilon}$, and so
$p\leq \epsilon D$.
We
therefore obtain
\[\|\partial c_3\|_1 \leq \epsilon t_{n-1} q.\]

Projecting $c_3$ back down to $M$ gives a chain $c_4$ with $[c_4]=
q[ M,\partial M]$ in relative homology, and
with 
\[\|c_4\|_1\leq \|c_3\|_1\leq\|c_2\|_1=
q\|c_1\|_1, \]
and with 
\[\|\partial c_4\|_1\leq \|\partial c_3\|_1 \leq
\epsilon t_{n-1} q.\]  
Dividing by $q$ we obtain a fundamental cycle $c = \frac{1}{q}c_4$
for $[ M,\partial M]$ satisfying $\|c\|_1 \leq
\| M,\partial M\|+2\epsilon$ and $\|\partial c\|_1\leq t_{n-1}\epsilon$.

Now we will take this relative cycle $c$ for $[ M,\partial M]$, and
modify it to an honest cycle in $M(T_1,\ldots,T_m)$.  We start by
pushing it forward from $M$ to $M(T_1,\ldots,T_m)$ by a map of pairs
\[h\co (M,\partial M)\to
 (M(T_1,\ldots,T_m),Z),\]
where $Z$ is the singular set of $M(T_1,\ldots,T_m)$ and $h$
is the quotient map described in Definition \ref{d:filling}.
(Here we are implicitly identifying $M$ with the space $\overline{M}$
described before Definition \ref{d:filling}.)
Let
$\partial h$ be $h$ restricted to $\partial M$.  The chain 
$(\partial h)_\sharp(\partial c)$ is an $(n-1)$--cycle in $Z$.  Lemma
\ref{l:thurstonlemma}
implies  there is some $K$ depending only on $n$ 
so that 
$(\partial h)_\sharp(\partial c)$ can be filled in $Z$ with a
$n$--chain $c'$  satisfying 
$\|c'\|_1\leq K \|(\partial h)_\sharp(\partial c)\|_1
\leq K t_{n-1}\epsilon$.
It is easy to check that $h_\sharp c-c'$ is a fundamental
cycle for $M(T_1,\ldots,T_m)$.  Moreover, we have
\[ \|h_\sharp c-c'\|_1 \leq \| M,\partial M\|+(2+Kt_{n-1})\epsilon. \]
Letting $\epsilon$ tend to zero, we have established
\[ \|M(T_1,\ldots,T_m)\| \leq \| M,\partial M\|.\]
An application of the proportionality theorem for finite volume
hyperbolic manifolds (Theorem \ref{t:proportionality}) gives the upper
bound:
\[ \|M(T_1,\ldots,T_m)\| \leq \frac{\Vol(V)}{v_n}.\]
\end{proof}

\appendix
\section{A proportionality theorem}\label{s:prop}
The purpose of this appendix is to establish the following
proportionality theorem.
\begin{theorem}\label{t:proportionality}
  Let $M$ be a manifold with boundary, so that the interior $V$
  of $M$ admits a complete hyperbolic metric of volume $\Vol(V)<\infty$.
  Then 
  \[ \|M,\partial M\| = \frac{\Vol(V)}{v_n}. \]
\end{theorem}

Thurston gives a proof in dimension $3$ of a version of theorem
\ref{t:proportionality} in his lecture notes \cite[6.5.4]{Th}, where
the norm $\|M,\partial M\|$ is replaced by one coming from measure
homology.  There is a natural map from singular homology to
measure homology, which at least does not increase norm.
Zastrow
and Hansen independently showed this map to be an isomorphism of
vector spaces for any CW pair
\cite{Han98,Zas98}, leaving open the question of whether it was an
isometry. 
L\"oh \cite{Loh06} proved
that the  map from (absolute) singular homology to measure homology
is an isometric isomorphism for all connected CW complexes.
Since it is not entirely clear whether measure homology is
isometric to singular homology in the relative case, we give a proof
of Theorem \ref{t:proportionality}
which avoids measure homology.  
Such a proof can also be obtained as a special (and simpler) case of
the arguments in 
Frigerio--Pagliantini \cite{FrPa}.  
Some related ideas may be found
in Francaviglia \cite{Fr04}.

Our proof of Theorem \ref{t:proportionality}
follows closely Benedetti and Petronio's proof of
\cite[C.4]{BP92}, which covers the case of
$\partial M = \emptyset$.  As some parts of our proof are identical to
steps in \cite{BP92}, we refer to that text for some details.
The proof in \cite{BP92} uses ordinary
real singular chains, rather than symmetrized ones, so we do not
symmetrize our chains here.

We use the following notation for parts of the manifold $V$:
For $I\subseteq \bR_{>0}$, we let $V_I$ be
the subset of $V$ consisting of points $x\in V$ where the injectivity
radius $\mathrm{inj}(V,x)$ is in $I$.
\begin{definition}
  Let $M$ be a Riemannian $n$--manifold, and 
  let $\sigma\co \Delta^n \to M$ be a smooth singular simplex.  The
  \emph{algebraic volume} of $\sigma$, written $\algvol(\sigma)$ is the
  integral over $\Delta^n$ of the pullback of the volume form on $M$.
  The algebraic volume is extended to smooth chains by linearity.
\end{definition}

Obtaining the correct lower bound for simplicial volume is fairly
easy.  
\begin{lemma}\label{l:lowerbound}
  $\|M,\partial M\|\geq \frac{\Vol(V)}{v_n}$.
\end{lemma}
\begin{proof}
  This direction is Thurston's ``straightening'' argument \cite[6.5.4]{Th}.
  For any sufficiently small $\epsilon>0$, the subset $V_{(0,\epsilon]}$
  is an open neighborhood of the cusps of $V$.  Collapsing the
  components of $V_{(0,\epsilon]}$ to their boundaries gives a homotopy
  equivalence of pairs $(V,V_{(0,\epsilon]}) \to (M,\partial M)$.  
  Since the Gromov norm of a (relative) homology class is homotopy
  invariant,  we can use chains $(V,V_{(0,\epsilon]})$ to compute
  $\|M,\partial M\|$.  

  The straightening map (see \cite{BP92} or \cite{Rat94} for a precise
  definition) is a chain map
  \[ \mathrm{str}\co C_*(V) \to C_*(V), \]
  chain homotopic to the identity, and
  taking each singular simplex to a totally geodesic simplex with the
  same vertices. 
  The map $\mathrm{str}$
  preserves the subspace $C_*(V_{(0,\epsilon]})$ (because horoballs
  are convex), as does the chain homotopy between $\mathrm{str}$ and
  the identity, so $\mathrm{str}$ induces a (norm-decreasing)
  chain homotopy equivalence of $C_*(V,V_{(0,\epsilon]})$ to itself.
  It follows that we need only consider straight (and therefore
  smooth) chains.

  If $z = \sum_i^k \lambda_i \sigma_i$ is a real smooth
  chain representing the relative fundamental class
  $[V,V_{(0,\epsilon]}]$, then
  \[ \algvol(z) = \sum_i^k \lambda_i\ \algvol(\sigma_i). \]
  The absolute value $|\algvol(z)|$ is at least as big as
  $\Vol(V_{[\epsilon,\infty)})$, the volume of the thick part of $V$.

  Let $z_\epsilon=\sum_i^k \lambda_i \sigma_i$ be a straight real
  singular chain representing 
  $[V,V_{(0,\epsilon]}]$ which $\epsilon$--nearly realizes
  $\|M,\partial M\|$, i.e., so that
  \[ \|z_\epsilon\|_1 \leq \|M,\partial M\|+\epsilon. \]
  We then have
  \begin{equation}\label{volepsilon}
    \Vol(V_{[\epsilon,\infty)})\leq |\algvol(z_\epsilon)|\leq
      \sum_{i=1}^k|\lambda_i|v_n = \|z_\epsilon\|_1 v_n \leq (\|M,\partial M\|+\epsilon)v_n.
  \end{equation}
  Letting $\epsilon$ tend to zero in \eqref{volepsilon} yields the lemma. 
\end{proof}
It is worth noting that essentially the same proof establishes the
bound of Lemma \ref{l:lowerbound} for measure cycles.

To obtain the upper bound for simplicial volume, we must construct
chains representing the fundamental class which are ``close'' to
the smeared chains from measure homology.  The argument from
\cite{BP92} in the closed case uses rather strongly that there
is a compact (and therefore finite \emph{diameter}) fundamental
domain.  We work around this by chopping the fundamental domain into
pieces of bounded diameter.  

\begin{proposition}\label{p:upperbound}
 $ \|M,\partial M\| \leq \frac{Vol(V)}{v_n}$
\end{proposition}
\begin{proof}
We suppose that $V$ is an orientable finite volume 
hyperbolic $n$--manifold with $m$ cusps, homeomorphic to the interior
of $M$.
Let $\Gamma < \mathrm{Isom}(\bH^n)$ be the fundamental group of $V$,
so that $\bH^n/\Gamma = V$.
Let $D$ be the closure of a convex fundamental domain
for the action of $\Gamma$ on $\bH^n$, and let $\pi\co D\to V$ be
the quotient.

There is a number $\mu>0$ so that $V_{(0,\epsilon]}$ has exactly $m$
connected components for any $\epsilon\leq \mu$.  Let $P_0\subseteq V$
be equal to $V_{(\mu,\infty)}$, and choose some $x_0\in P_0$.  The
complement of $P_0$ in $V$ is a union of cusp neighborhoods
$C_1,\ldots,C_n$.

Let $D_0 = \pi^{-1}(P_0)\subset D$, and choose $\tilde{x}_0$ in the
interior of $D_0$.  Let $x_0 = \pi(\tilde{x}_0)\in P_0$.

For each $l\in \bN$, and each $j\in \{1,\ldots,m\}$, we let  
\[ P_{j,l} = \{x\in V\mid l-1\leq d(x,P_0)\leq l\}\cap C_j, \]
and let $D_{j,l} = \pi^{-1}(P_{j,l})\subset D$.  
We will refer to any $\Gamma$--translate of $D_0$ or $D_{j,l}$ as a
\emph{piece}.  

There is some $d>0$ so that $\mathrm{diam}(D_\alpha)<d$ for every piece
$D_\alpha\subset\bH^n$. 
For each $l\in \bN$ and each $j\in \{1,\ldots,m\}$, choose $\tilde{x}_{j,l}$ in
the interior of
$D_{j,l}$ and let $x_{j,l}=\pi(\tilde{x}_{j,l})$.  
The union of the $\Gamma$--orbits of $\tilde{x}_0$ and the
$\tilde{x}_{j,l}$ can be 
identified in an 
obvious way with 
\[
\hat{\Gamma}:= \Gamma\cup\left(\Gamma\times\{1,\ldots,m\} \times \bN\right).
\]
(This is essentially the $0$--skeleton of the ``cusped space'' of
\cite{GM} associated to the pair 
$\left(\Gamma,\{\pi_1(C_j)\}\right)$ .)
Vertices of $\hat{\Gamma}$ have a \emph{depth} associated to them:
The depth of an element of $\Gamma$ is $0$, and the depth of an
element of $\Gamma\times\{1,\ldots,m\}\times \bN$ is given by the third
coordinate.  
The
group $\Gamma$ acts
on $\hat{\Gamma}$ by left multiplication in an
obvious way, and this action preserves depth.

We let $\Omega$ be the quotient
$\raisebox{-2pt}{$\Gamma$}\!\!\leftquotient\!\hat{\Gamma}^{n+1}$.
For $k>0$ an integer, let $\hat{\Gamma}_k\subset \hat{\Gamma}$ be the
subset of vertices of 
depth at most $k$.  
Let $\tilde{\Omega}_k = \hat{\Gamma}_k^{n+1}$, and
let $\Omega_k\subset \Omega$ be the set of $\Gamma$--orbits in
$\tilde{\Omega}_k$. 

For any element $\omega=[(y_0,\ldots,y_n)]\in \Omega$, there is a
unique straight
simplex 
\[\sigma_\omega\co \Delta^{n}\to V\] 
which is equal to the composition of 
a simplex $\tilde{\sigma}_\omega\co
\Delta^{n}\to\bH^n$ satisfying $\tilde{\sigma}_\omega(v_i) = y_i$ with
the projection $\bH^n\to V$.  (Here $v_i$ is the $i$th vertex of
the $n$--simplex $\Delta^n$.) 
As in \cite{BP92}, we define
\[ \mathcal{S}(R) = \{(u_0,\ldots u_n)\in (\bH^n)^{n+1}
    \mid d(u_i,u_j)=R,\ \forall i\neq j\}
    = \mathcal{S}_+(R)\sqcup\mathcal{S}_-(R),\]
where $\mathcal{S}_+$ and $\mathcal{S}_-$ are those tuples which
determine positively and negatively oriented straight $n$--simplices in
$\bH^n$, respectively.  There is a measure on $\mathcal{S}(R)$ coming
from Haar measure on $\mathrm{Isom}(\bH^n)$, and we denote this
measure by $m$.  For 
\[\omega=\left[\left(y_0,\ldots,y_n\right)\right]
\in \Omega,\] 
and $R>0$, we define
\begin{align*}
 a_R^+(\omega)& = m \left(\left\{ 
   (u_0,\ldots,u_n)\in \mathcal{S}_+(R)\mid u_i\mbox{ and }y_i\mbox{ lie in
   the same piece}
 \right\}\right)\mbox{, and}\\
 a_R^-(\omega)& = m \left(\left\{ 
   (u_0,\ldots,u_n)\in \mathcal{S}_-(R)\mid u_i\mbox{ and }y_i\mbox{ lie in
   the same piece}
 \right\}\right).
\end{align*}
Obviously these definitions do not depend on the choice of
representative of $\omega$.  Moreover, for a given $R$ and $k$, there
are only finitely many $\omega\in \Omega_k$ for which $a_R^+(\omega)$
or $a_R^-(\omega)$ is nonzero. (This is because there are only
finitely many elements of $\hat{\Gamma}$ which are distance
approximately $R$ from a given element of $\hat{\Gamma}$.)
It follows that
if we define
$a_R(\omega) := a_R^+(\omega) - a_R^-(\omega)$, then
\[ z_{R,k} = \sum_{\omega\in \Omega_k} a_R(\omega)\sigma_\omega \]
is a finite real singular $n$--chain in $V$.  Moreover, we claim that 
\begin{equation}\label{zinfty}
 z_{R,\infty} = \sum_{\omega\in\Omega} a_R(\omega)\sigma_\omega 
\end{equation}
is a \emph{locally finite} real singular $n$--chain.  Indeed, 
let $p\in V$, and $r>0$.  We must show only finitely many simplices in
the support of $z_{R,\infty}$ (i.e.,
occurring in
the sum \eqref{zinfty} with nonzero coefficient)
intersect the ball of
radius $r$ about $p$.  Let $\tilde{p}$ be any point in
$\pi^{-1}(p)\subset D$,  and note that any simplex $\sigma_\omega$ in
the support of $z_{R,\infty}$ must lift to one which comes within $r$
of $\tilde{p}$, and whose vertices are therefore within $R+2d+r$ of
$\tilde{p}$.  As there are only finitely many such vertices, there are
only finitely many such simplices $\sigma_\omega$ in the support of
$z_{R,\infty}$. 
\begin{claim}\label{c:cycle}
  The chain $z_{R,\infty}$ is a locally finite cycle.
\end{claim}
\begin{proof}
  We have just argued that $z_{R,\infty}$ is locally finite.  The
  proof that $z_{R,\infty}$ is a cycle is nearly the same as the
  proof of Claim (ii) on pp. 115--116 of \cite{BP92}, except that
  sums/unions over $\Gamma$ are replaced by sums/unions over $\hat{\Gamma}$, and
  conditions of the form ``$u_i\in \gamma_i(D)$'' are replaced
  by ``$u_i \in D(i)$'' for appropriately chosen pieces $D(i)$.
\end{proof}

For $k\in \bN$, let $\epsilon(k)$ be the smallest injectivity radius at
any point in a piece of depth at most $k$, i.e.,
\[ \epsilon(k) = \inf\left\{\mathrm{inj}(V,x)\ \left|\  x\in
P_0\mbox{ or }x\in\bigcup_{}
\left\{P_{j,l}\mid  j\in\{1,\ldots,m\},
l\leq k 
\right\}
\right\}\right..\]
It is not hard to see that as $k\to \infty$, the quantity
$\epsilon(k)$ tends to zero.
In particular, there is a constant $C$
so that if $k>C$, then $\epsilon(k)<\mu$.  
\begin{claim}\label{c:relcycle}
  Let $R>0$, and let $C$ be such that $\epsilon(k')<\mu$ for all $k'>C$.
  If $k>R+C+2d+1$, then $z_{R,k}$ is a relative cycle in $(V,V_{(0,\mu]})$.
\end{claim}
\begin{proof}
  Suppose an $(n-1)$--dimensional simplex $\sigma$ is in the support
  of  $\partial z_{R,k}$.  By Claim \ref{c:cycle}, $\sigma$ must also be in the
  support of $\partial(z_{R,\infty}-z_{R,k})$, so any lift
  $\tilde{\sigma}$ of $\sigma$ to $\bH^n$ must have a vertex mapped to
  $\hat{\Gamma}\smallsetminus\hat{\Gamma}_k$.  The 
  diameter of the image of $\tilde\sigma$ is at most $R+2d$.  Thus the
  vertices lie in
  $\hat{\Gamma}\smallsetminus \hat{\Gamma}_{k'}$ for 
  $k' = k-\lceil R+2d \rceil>C$, and no part of the image of
  $\tilde\sigma$ lies in a translate of the ``fat piece'' $D_0$.
  Since $\epsilon(k')<\mu$, $\sigma$ has
  image in the part of $V$ with injectivity radius less than $\mu$.
\end{proof}
\begin{claim}\label{c:claimiii}
  If $R>2d$ then $a_R^+(\omega)\cdot a_R^-(\omega) = 0$ for all
  $\omega\in \Omega$.
\end{claim}
\begin{proof}
  This statement (and its proof) are identical to Claim (iii) on
  pp. 116--117 of \cite{BP92}.
\end{proof}

\begin{claim}\label{c:volume}
  There is a quantity $\delta(R)>0$ so that $\lim_{R\to\infty}\delta(R)
  = 0$ and
 so that the
  simplices in the support of $z_{R,\infty}$ all have volume at least
  $v_n-\delta(R)$.
\end{claim}
\begin{proof}
  This statement is a slight rephrasing of Claim (iv) on p. 117 of
  \cite{BP92}.   The proof is the same.
\end{proof}

\begin{claim}\label{c:positive}
  If $R>2d$ and $a_R(\omega) \neq 0$, then
  $a_R(\omega)\algvol(\sigma_\omega)>0$. 
\end{claim}
\begin{proof}
  See Claim (v) on p. 117 of \cite{BP92}. 
  Claim \ref{c:claimiii} is used here.    
\end{proof}

\begin{claim}\label{c:nontrivial}
  Suppose $R>2d+C$ and 
  $k>2R+1$, where $C$ is the constant from Claim \ref{c:relcycle}.
  The classes $[z_{R,\infty}]\in H_n^{lf}(V;\bR)$ and
  $[z_{R,k}]\in H_n(V,V_{(0,\mu]};\bR)$ are nontrivial.
\end{claim}
\begin{proof}
  See Claim (vi) on p.117 of \cite{BP92}.  The point here is just that
  some element of $\mathcal{S}(R)$ is a tuple of interior points of
  pieces corresponding to some tuple $\omega$, so that there is an
  open (and therefore positive measure) neighborhood of tuples of
  points all corresponding to the same tuple $\omega$.
\end{proof}
We now complete the proof of the Proposition.
Suppose $R>2d+C$ and $k = \lceil 2R \rceil+2$.

It follows from Claim \ref{c:nontrivial} that there is
some $\lambda>0$ so that $\lambda[z_{\infty,R}]$ is the fundamental
class in the locally finite homology group
$H_n^{lf}(V;\bR)$ and some $\lambda'$ so
$\lambda'[z_{k,R}]$ is the fundamental
class in $H_n(V,V_{(0,\epsilon]};\bR)$.
In fact we must have $\lambda'=\lambda$, since both
$\lambda[z_{\infty,R}]$ and $\lambda'[z_{k,R}]$ restrict to the
orientation cycle at $x_0$, and the 
set of simplices in the supports of $z_{k,R}$ and $z_{\infty,R}$
which intersect a small neighborhood
of $x_0$ is the same.

It follows from Claim \ref{c:positive} that 
$\algvol(z_{k,R})\leq\algvol(z_{\infty,R})$.  Moreover, an easy
argument shows that if $c$ is a locally finite fundamental cycle for
$V$, then $\Vol(V) = \algvol(c)$.  We therefore have:
\begin{align*}
  \Vol(V) & 
  \geq  \algvol(\lambda z_{k,R})& \\
 &  =  \sum_{\omega\in\Omega_k} \lambda a_R(\omega) \algvol(\sigma_\omega)& \\
 & =  \sum_{\omega\in\Omega_k} \lambda |a_R(\omega)|\cdot
  |\algvol(\sigma_\omega)|& \mbox{ by \ref{c:positive}}\\
 & \geq  (v_n - \delta(R))\sum_{\omega\in \Omega_k}\lambda |a_R(\omega)| & \mbox{ by
    \ref{c:volume}}\\
 & =  (v_n - \delta(R))\|\lambda z_{k,R}\|_1 \geq (v_n-\delta(R))\|M,\partial
  M\|, & 
\end{align*}
since $\lambda z_{k,R}$ gives a representative for the fundamental class
of $(M,\partial M)$ via the homotopy equivalence of pairs
$(V,V_{(0,\mu]})\to (M,\partial M)$.  
Letting $R$ tend to infinity, Claim \ref{c:volume} says that
$\delta(R)$ tends to zero, so we have
\[ \Vol(V) \geq v_n\|M,\partial M\|, \]
and the proposition is proved.
\end{proof}

Lemma \ref{l:lowerbound} and Proposition \ref{p:upperbound} together
immediately imply Theorem \ref{t:proportionality}.

\section{(Anti)Symmetrization of chains}\label{a:symm}
In this section we prove that the chain map $S$ defined in
Section \ref{ss:symm} is  chain homotopic to the identity.

We let $\Delta^n$ denote the standard $n$--simplex, and begin by
defining a ``coning'' operator for all $n$ and $k$
\[ c_n\co C_k(\Delta^n;\bR)\to C_{k+1}(\Delta^n;\bR) .\]
We define $c_n$ on a singular $k$--simplex $\psi\co \Delta^k\to
\Delta^n$ and then extend by linearity.  
Let $\iota\co \Delta^k\to \Delta^{k+1}$ be the affine map taking each
vertex $e_i$ of $\Delta^{k}$ to the vertex $e_{i+1}$ of
$\Delta^{k+1}$; the vertex $e_0$ is the only one missed by $\iota$.
Each point in
$\Delta^{k+1}$ is uniquely expressible convex combination of some point in $\iota(\Delta^{k})$
with $e_{0}$.
Let $m_n$ be the barycenter
of $\Delta^n$, and define 
\[ c_n(\psi)\left(t e_0 + (1-t)\iota(x)\right) = t m_n + (1-t)
\psi(x). \]
In words, $c_n(\psi)$ is $\psi$ coned to the barycenter $m_n$.  The
key fact about $c_n$ is that for any chain $a\in C_k(\Delta^n;\bR)$,
we have
\begin{equation}\label{chcon}
  \partial c_n(a) = a - c_n(\partial a).
\end{equation}
(In other words, $c_n$ is a chain contraction of $C_*(\Delta^n;\bR)$.)

We use the coning operator to inductively
define an element
$\mu_n\in C_{n+1}(\Delta^n;\bR)$,
and, for any space $Y$, a map
\[P_{n}^{(Y)} \co C_{n}(Y;\bR)\to C_{n+1}(Y;\bR).\]
As in the definition of symmetrization, we abuse notation by using $q\in S_{n+1}$ to denote the
unique affine map from $\Delta^n$ to itself which is equal to $q$ when
restricted to the vertices.  We start with $\mu_0 = 0$, and 
$P_0 = 0$ for all $Y$.  
We may assume $P_{n-1}^{(\Delta^n)}$ has been
defined, and set
\begin{equation}\label{mudef}
\mu_n =     \left(\frac{1}{|S_{n+1}|} \sum_{q\in
      S_{n+1}}\sign(q) c_n(q)\right) 
      - c_n(\Id_{\Delta^n})
      - c_n\left(P_{n-1}^{(\Delta^n)}(\partial \Id_{\Delta^n})\right).
\end{equation}
For any space $Y$, we define $P_n\co C_{n}(Y;\bR)\to C_{n+1}(Y;\bR)$
on  simplices, and then extend linearly; for $\phi\co\Delta^n\to Y$ a
singular $n$--simplex, define
\begin{equation}\label{Pdef}
  P_n^{(Y)}(\phi) =  \phi_\sharp(\mu_n).
\end{equation}
The maps $P_n^{(\cdot)}$ commute with continuous maps:
\begin{lemma}\label{l:pfunk}
  If $f\co X\to Y$ is any continuous map, and $n$ is an integer, then
  $f_\sharp \circ P_n^{(X)} = P_n^{(Y)}\circ f_\sharp$.
\end{lemma}
\begin{proof}
  Since all the maps involved are linear, it suffices to verify this
  for a single singular simplex $\phi\co \Delta^n\to X$.  We have:
  \begin{align*}
    (f_\sharp \circ P_n^{(X)})(\phi) & =  f_\sharp\phi_\sharp(\mu_n) 
      = (f\circ\phi)_\sharp(\mu_n) = P_n^{(Y)}(f \circ \phi) =
      (P_n^{(Y)}\circ f_\sharp)(\phi).
  \end{align*}
\end{proof}

\begin{lemma}
  $P_*^{(X)}$ as defined in equation \eqref{Pdef} is a chain homotopy between
  $S$ and $\Id_{C_*(X;\bR)}$.
\end{lemma}
\begin{proof}
  In this proof we'll write $P_k^{(X)}$ as $P_k$.

  It suffices to show that, for each $n$ and each singular $n$--simplex $\phi$,
  \[ \partial P_n(\phi) = S(\phi) - \phi - P_{n-1}(\partial \phi). \]
  For $n= 0$, this is immediate, so we suppose that $n>0$ and argue inductively.
  We compute from \eqref{Pdef} and \eqref{chcon},
  \begin{align*}
  \partial P_n(\phi) = & \phi_\sharp\Bigg(
  \frac{1}{|S_{n+1}|}\sum_{q\in S_{n+1}} \sign(q) (q - c_n(\partial q))\\
   & \quad\quad  - \Id_{\Delta^n} + c_n(\partial \Id_{\Delta^n}) 
     - P_{n-1}(\partial \Id_{\Delta^n}) + c_n \partial \left(P_{n-1}^{(\Delta^n)}(\partial \Id_{\Delta^n})\right)\Bigg)
  \end{align*}
By induction on $n$ (and the fact that $\partial^2 = 0$) we have
\[ \partial P_{n-1}^{(\Delta^n)}(\partial \Id_{\Delta^n}) = S(\partial
\Id_{\Delta^n}) - \partial \Id_{\Delta^n}, \]
so we can rewrite
\begin{align*}
\partial P_n(\phi) &  =  \phi_\sharp\Bigg( 
\frac{1}{|S_{n+1}|}\sum_{q\in S_{n+1}} \sign(q) (q - c_n(\partial q))\\
   &  \quad\quad\quad - \Id_{\Delta^n} + c_n(\partial \Id_{\Delta^n}) 
     - P_{n-1}^{(\Delta^n)}(\partial \Id_{\Delta^n}) + c_n (S(\partial\Id_{\Delta^n}) -
     \partial \Id_{\Delta^n}) \Bigg)\\
 & = \phi_\sharp\Bigg(S(\Id_{\Delta^n}) - \Id_{\Delta^n} - P_{n-1}^{(\Delta^n)}(\partial \Id_{\Delta^n})\\
 &   \quad\quad\quad  + c_n(S(\partial \Id_{\Delta^n} )) - \frac{1}{|S_{n+1}|}\sum_{q\in S_{n+1}} \sign(q) c_n(\partial
     q) \Bigg)
\end{align*}
The last two terms exactly cancel because $c_n$ and $\partial$ are both linear, so
\[\frac{1}{|S_{n+1}|}\sum_{q\in S_{n+1}} \sign(q) c_n(\partial
     q) = c_n\partial \left( \frac{1}{|S_{n+1}|}\sum_{q\in S_{n+1}}
     \sign(q) q\right) = c_n\partial S(\Id_{\Delta^n})\]
and, since $S$ is a chain map, $c_n(\partial S(\Id_{\Delta^n}))=c_n(S(\partial \Id_{\Delta^n})) $.
We therefore have
\[ \partial P_n(\phi) = \phi_\sharp\left(S(\Id_{\Delta^n}) -
\Id_{\Delta^n} - P_{n-1}^{(\Delta^n)}(\partial \Id_{\Delta^n})\right).  \]

We now invoke Lemma \ref{l:pfunk}, which implies that
\begin{equation*}\label{differentP}
\phi_\sharp\left(P_{n-1}^{(\Delta^n)}(\partial \Id_{\Delta^n})\right)
=P_{n-1}(\partial \phi),
\end{equation*}
and so
\[ \partial P_n(\phi)= S(\phi) -
\phi - P_{n-1}(\partial \phi).\]
\end{proof}

\section*{Acknowledgments}
We thank Danny Calegari,  Mike Davis, Bill Thurston, Clara L\"oh,
and Shmuel Weinberger for useful
conversations.   We also thank Alexander
Nabutovsky for pointing out to us
that our work applies to the homology spheres of Ratcliffe and
Tschantz. 

Jason Manning did part of this work during visits to
the Caltech mathematics department and 
the Cornell mathematics department, and
thanks Caltech and Cornell for their hospitality.

\end{document}